\documentclass{article}

\usepackage{amsmath}
\usepackage{amsthm}
\usepackage{amsfonts}
\usepackage{indentfirst}

\newtheorem{theorem}{Theorem}[section]
\newtheorem{lemma}[theorem]{Lemma}

\newtheorem{corollary}[theorem]{Corollary}

\newtheorem{definition}[theorem]{Definition}

\newtheoremstyle{example}{\topsep}{\topsep}%
     {}
     {}
     {\bfseries}
     {.}
     {.5em}
     {\thmname{#1}\thmnumber{ #2}\thmnote{ #3}}

\theoremstyle{example}
\newtheorem{remark}[theorem]{Remark}
\newtheorem{example}[theorem]{Example}

\newcommand{\set}[1]{\left\{ #1\right\}}
\renewcommand{\vec}[1]{\textbf{#1}}

\begin{document}

\title{Weight recursions for any rotation symmetric Boolean functions}
\author{Thomas W. Cusick \footnote{University at Buffalo, Buffalo, NY, USA;  e-mail: cusick@buffalo.edu}}
\date{}

\maketitle
\begin{abstract}
Let $f_n(x_1, x_2, \ldots, x_n)$ denote the algebraic normal form (polynomial form) of a rotation symmetric
Boolean function of degree $d$ in $n \geq d$ variables and let $wt(f_n)$ denote the Hamming weight of this function.
Let $(1, a_2, \ldots, a_d)_n$ denote the function $f_n$ of degree $d$ in $n$ variables generated by the monomial
$x_1x_{a_2} \cdots x_{a_d}.$ Such a function $f_n$ is called {\em monomial rotation symmetric} (MRS). It was proved in a $2012$ paper that for any MRS $f_n$ with $d=3,$ the sequence of weights $\{w_k = wt(f_k):~k = 3, 4, \ldots\}$ satisfies a homogeneous linear recursion with integer coefficients.  In this paper it is proved that such recursions
exist for any rotation symmetric function $f_n;$ such a function is generated by some sum of $t$ monomials of various degrees.  The last section of the paper gives a Mathematica program which explicitly computes the homogeneous linear  recursion for the weights, given any rotation symmetric $f_n.$   The reader who is only interested in finding some recursions can use the program and not be concerned with the details of the rather complicated proofs in this paper.   
\end{abstract}



\section{Introduction}
\label{intro} 
If we define $V_n$ to be the vector space of dimension $n$ over the
finite field  $GF(2) = \set{0, 1}$, then an $n$ variable Boolean function  $f(x_1, x_2, ..., x_n)
= f(\vec{x})$ is a map from $V_n$ to $GF(2)$.  Every Boolean function $f(\vec{x})$ has a unique
polynomial representation (usually called the {\em algebraic normal form} \cite[p. 6]{CBF}),
and the {\em degree} of  $f$ (notation $deg ~f$) is the degree of this polynomial.
A function of degree $\leq 1$ is \emph{affine}, and if the constant term is 0, then the function is \emph{linear}.
We let  $B_n$ denote the set of all Boolean functions in $n$ variables, with addition
and multiplication done $\bmod ~{2}.$  When addition $\bmod~{2}$ is clear from the context we use $+,$ 
but if we wish to emphasize the fact that addition is being done $\bmod~{2}$ we will use 
$\oplus.$ We also use $\oplus$ for the xor addition of two binary $m$-tuples.  We use 
$a \| b$ to denote the concatenation of two strings $a$ and $b.$

If we list the $2^n$ elements of $V_n$ as $v_0 = (0,\ldots,0), v_1 = (0,\ldots,0,1), \ldots$ in
lexicographic order, then the $2^n$-vector $(f(v_0), f(v_1),\ldots,f(v_{2^n - 1}))$ is
called the {\em truth table} of $f$.
The {\em weight} (also called Hamming weight)  $wt(f)$
of $f$ is defined to be the number of 1's in the truth table for $f$.
In many cryptographic uses of Boolean functions, it is important that the truth table of each function  $f$  has an equal number of 0's and 1's; in that case, we say that the function $f$  is {\em balanced}.

We define a cyclic permutation $\rho$ on $n$ variables by 
$\rho(x_1, x_2, \cdots, x_n)=(x_2, x_3, \cdots, x_n, x_1)$. Then a Boolean function $f(x)$ in $n$ variables, where $x=(x_1, x_2, \cdots, x_n)$, is  \emph{rotation symmetric} if $f(x)=f(\rho(x))$ for all $x \in V_n$.  A Boolean function is \emph{monomial rotation symmetric} (MRS) if it is rotation symmetric and generated by a single monomial.  In \cite{PQ} Piepryzyk and Qu showed that rotation symmetric Boolean functions are useful  in cryptography for designing fast hash functions.  Since then, further applications of these functions in cryptography and coding theory have been found (many references for this are in \cite{CBF}), so much attention has been given to rotation symmetric Boolean functions.  

A summary of the work on rotation symmetric functions is given in \cite[Chapter 6]{CBF}.
This paper is concerned with the following theorem.  We need some notation first.  We use the notation 
$(1, a_2, \ldots, a_d)_n$ for the monomial rotation symmetric function
$f(x_1, x_2, \ldots, x_n)$ of degree $d$ in $n$ variables generated by the monomial
$x_1x_{a_2} \cdots x_{a_d}.$

\begin{theorem}
\label{thm1}
Let $f_k(x_1, x_2, \ldots, x_k)$ denote any rotation symmetric function in $k$ variables generated by the sum of $t$ monomials 
$x_1x_{a_{2,i}} \cdots x_{a_{d(i),i}}, ~ 1 \leq i \leq t,$ where $d(i)$ is the degree of the $i$-th monomial.  Then the sequence of weights $\{w_k = wt(f_k):~k = max~ d(i),  max~ d(i) +1, \ldots\}$
satisfies a homogeneous linear recursion with integer coefficients.  
\end{theorem}

We note that in the simplest special case of Theorem \ref{thm1} where $t=1$ and $f_k$ is quadratic one does not even need to consider recursions, since in that case \cite[Th. 8, p. 431]{Kim09} the weights $wt(f_k)$ themselves are given by a concise formula.  Thus there is no loss of generality in assuming max $d(i) \geq 3$ in Theorem \ref{thm1}.

The first faint hint for Theorem \ref{thm1} was given in \cite[Th. 10, p. 297]{CS02}, where a 
nonhomogenous linear recursion for the sequence of weights $\{wt(g_n): n = 3, 4, \ldots \}$
with $g_n(x_1,x_2,x_3) = (1,2,3)_n$ was given.  This recursion could easily  be converted into a homogeneous one, but the actual statement of the homogeneous recursion was not published until the following example \cite[Section 6.4.3, p. 124]{BCP}.

\begin{example}
\label{ex1}
 The sequence of weights $\{w_n = wt(g_n): n = 3, 4, \ldots \}$
with $g_n(x_1,x_2,x_3) = (1,2,3)_n$ satisfies the recursion
$$ w_n = -2w_{n-1}-2w_{n-2}+2w_{n-3}+4w_{n-4} = 0,~n = 7, 8 \ldots~.$$
of order $4.$ Therefore the {\em recursion polynomial} is $x^4 + 2x^3 + 2x^2 -2x - 4.$
\end{example}

The subsequent history leading up to Theorem \ref{thm1} is as follows.  The special case where $t=1$ and $f_k$ is a cubic MRS function was proved in 
\cite[Section 6]{BCP}.  The method used in that proof was simplified in \cite{BC12}, though the proof is still complicated.  By extending the ideas in \cite{BC12}, we obtain the proof of Theorem \ref{thm1} in this paper.  The details of this proof are intricate, but in Section \ref{prog} of this paper we give a Mathematica program which explicitly computes the homogeneous linear  recursion whose existence is proved in Theorem \ref{thm1}.  The reader who is only interested in finding some recursions can simply use the program and not be concerned with the details of the proof of Theorem \ref{thm1}.    

\section{Sum of cubic MRS and quadratic MRS}
\label{deg3+deg2}
In this section, we give the proof of one of  the simplest cases of Theorem \ref{thm1} where $t > 1,$ namely the case where $f_k$ is generated by the sum of a cubic and a quadratic monomial. Familiarity with the proofs in \cite{BC12} is assumed, since the proofs here build on the ideas in that paper, and we also frequently use results and notation from that paper.
Our goal in this section is to prove Theorem \ref{thm1} for the function $(1,r,s)_n + (1,a)_n.$

\subsection{Preliminaries and g-functions}
\label{subsec2.1}
\begin{lemma}
\label{quadTT}
Let $T_{n,s}$ denote the truth table of the monomial $x_1x_s$ in $n$ variables. Then we have
$$
T_{n,s} = 0_{2^{n-1}}(0_{2^{n-s}}1_{2^{n-s}})_{2^{s-2}} 
$$
where $a_k$ represents the integer $a\in \{0,1\}$ repeated $k$ times. 
\end{lemma}
\begin{proof}
This result (in a different notation) was given in  \cite[Lemma 8, p. 293]{CS02}. It is a very special case of Lemma \ref{generalTT}.  
\end{proof}

As in \cite{BC12}, it will be convenient to use a superscript to denote the number of variables in a Boolean function.  We also use the notation $T(f^k)$ for the truth table of a Boolean function $f^k$ in $k$ variables. 
\begin{definition}
For $v\in \mathbb{Z}$, $1\leq v \leq s$, we define an \emph{$\tilde{m}_v $-action} on a truth table of a quadratic Boolean function $f^{n-1}$ in $n-1$ variables as follows: Define $\tilde {m}_v$ to be the sequence obtained by splitting the truth table of $x_1x_s$ into $2^{s-v+1}$ equal sized portions, isolating the final portion, and then stretching it (by repeating each entry $2^{s-v}$ times) to a length of $2^{n-1}$. Then the $\tilde{m}_v$-action applied to $f^{n-1}$ is
$$
T(\overset{\tilde{m}_v}{f^{n-1}}) = T(f^{n-1}) \oplus \tilde{m}_v.
$$  
\end{definition}
\begin{lemma}
\label{quadraticactions}
For $1\leq v\leq s$, the $\tilde{m}_v$-actions are given by 

\begin{align*}
1.\hspace{.5pc} & m = 1_{2^{n-1}} &v=1\\
2.\hspace{.5pc} & m = (0_{2^{n-v}}1_{2^{n-v}})_{2^{v-2}}  & 1< v \leq s
\end{align*}
\end{lemma}
\begin{proof}
This follows at once from Lemma \ref{quadTT}.
\end{proof}
Now we consider rotation symmetric Boolean functions which are obtained by adding a cubic MRS function to a quadratic MRS function. Let $f_1$ be the cubic MRS function generated by $x_1x_rx_s$. Let $\tilde{f}_2$ be the quadratic MRS function generated by $x_1x_a$. In the following we will derive recursions for the truth tables and weights of the function $f=f_1 + \tilde{f}_2$. \\
We begin, as was done in  \cite{BC12}, by looking at the truncated $g$-functions which are defined by looking at only the first $1+n-s$ terms of $f_1$ or the first $1+n-a$ terms of $\tilde{f} _2$. That is,
\begin{align*}
g_1^n &= x_1x_{r}x_{s}+x_2x_{r+1}x_{s+1}+\ldots +x_{1+n-s}x_{r+n-s}x_n\\
\tilde{g}_2^n &=x_1x_a+x_2x_{a+1} + \ldots + x_{1+n-a}x_n
\end{align*}
We define a map $\sigma: B_n \rightarrow B_{n+1}$ which sends the $x_i$ to $x_{i+1}$. Applying $\sigma $ to $g_1$ and $\tilde{g}_2$ gives
\begin{align*}
\sigma (g_1^n) &= x_2x_{r+1}x_{s+1}+\ldots +x_{1+n-s}x_{r+n-s}x_n+x_{1+(n+1)-s}x_{r+(n+1)-s}x_{n+1}\\
\sigma (\tilde{g}_2^n) &= x_2x_{a+2} + \ldots + x_{1+n-a}x_n + x_{2+n-a}x_{n+1},
\end{align*}
with all indices reduced mod $n$ if necessary.
From the above, we can easily see that $g_1^{n+1} = \sigma (g_1^n) + x_1x_{r}x_{s}$ and that $\tilde{g}_2^{n+1} = \sigma (\tilde{g}_2^n) + x_1x_a$.  Thus the truth tables of $g_1^{n+1}$ and $\tilde{g}_2^{n+1}$ are given respectively  by 
\begin{align*}
T(g_1^{n+1}) &= T(\sigma(g_1^{n})) \oplus T(x_1x_{r}x_{s}), \\
T(\tilde{g}_2^{n+1}) &= T(\sigma(\tilde{g}_2^{n})) \oplus T(x_1x_a). 
\end{align*}
It is then easy to see that the truth table of $g^{n+1}=g_1^{n+1}+\tilde{g}_2^{n+1} $ is given by
$$
T(g^{n+1}) = T(\sigma(g_1^{n})) \oplus T(\sigma(\tilde{g}_2^{n})) \oplus T(x_1x_{r}x_{s}). \oplus T(x_1x_a)
$$
Notice that, since $\sigma (g_1^{n})$ does not contain $x_1$, the first and second halves of $T(\sigma(g_1^{n}))$ are identical. This also holds for $\sigma(\tilde{g}_2^{n})$, since it does not contain $x_1$ either. Further, we can see that $\sigma (g_1^{n})=g_1^{n+1}(0,x_1,\ldots , x_{n+1})$. By relabelling variables, we can see this to be equivalent to $g_1^{n}$. The same  result holds for $\tilde{g}_2$. Hence $\sigma (g)$ has identical left and right halves and they are equal to $g_1^n \oplus \tilde{g}_2^n$ and
$$
T(\sigma (g^n))= T(g_1^{n}) \oplus T(\tilde{g}_2^{n}) \| T(g_1^{n}) \oplus T(\tilde{g}_2^{n}). 
$$
Further, we note that, since $x_1=0$ on the first half of the truth table, $x_1x_{r}x_{s}$ and $x_1x_a$ only affect the second half of the truth table of $g^{n+1}$. Thus
$$
T(g^{n+1}) = T(g_1^{n}) \oplus T(\tilde{g}_2^{n}) \| T(g_1^{n}) \oplus T(\tilde{g}_2^{n}) \oplus T(x_1x_{r}x_{s} ) \oplus T(x_1x_a). 
$$
As was done in \cite{BC12}, we can represent the effects of $T(x_1x_{r}x_{s})$ and $T(x_1x_a)$ using $m$-actions and $\tilde{m}$-actions. (Recall, since $m$-actions depend on $r$ and $s$, we use $m_i$ to refer to $m$-actions that arise from the monomial $x_1x_{r}x_{s}$ and $\tilde{m} _i$ to refer to  $m$-actions that arise from the monomial $x_1x_a$). So we have 
$$
T(g^{n+1}) = T(g_1^{n}) \oplus T(g_2^{n}) \| \overset{m_{s}+\tilde{m} _{a}}{T(g_1^{n}) \oplus T(g_2^{n})}.
$$  
\begin{lemma}
\label{quadsplits}
For any $\tilde{m}$-action, $\tilde{m}_v$ where $s<n$ and $1\leq v \leq s$, we can split 
$\tilde{m}_v$ into two $\tilde{m}$-actions, $\tilde{m}_{v1}$ and $\tilde{m}_{v2}$ such that 
$\tilde{m}_{v1}$ acts on the first half of a truth table  $T(b^{n-1})$ of a Boolean function in $n-1$ variables and $\tilde{m}_{v2}$ acts on the second half as follows: 
\begin{align*}
1.\hspace{.5pc} & \overset{\tilde{m}_v^n}{T(b^{n-1})} = \overset{\tilde{m}_v^{n-1}}{T^1(b^{n-1})}\| \overset{\tilde{m}_v^{n-1}}{T^2(b^{n-1})} &v=1\\
2.\hspace{.5pc} & \overset{\tilde{m}_v^n}{T(b^{n-1})} = T^1(b^{n-1})\| \overset{\tilde{m}_{v-1}^{n-1}}{T^2(b^{n-1})} &v=2\\
3.\hspace{.5pc} & \overset{\tilde{m}_v^n}{T(b^{n-1})} = \overset{\tilde{m}_{v-1}^{n-1}}{T^1(b^{n-1})}\| \overset{\tilde{m}_{v-1}^{n-1}}{T^2(b^{n-1})} &otherwise\\
\end{align*}
where $T^1(b^{n-1})$ is the first half of $T(b^{n-1})$ and $T^2(b^{n-1})$ is the second half.
\end{lemma}
\begin{remark}
Note that these are the splitting rules given in \cite[Lemma 5]{BC12}, with the obvious exception that we do not need to account for $v=s-r+2$ in case 2 above. 
\end{remark}
\begin{proof}
This easily follows from Lemma \ref{quadraticactions}.
\end{proof}

The rules for the splitting of the $m$-actions are given in \cite[Lemma 5]{BC12}, which is exactly the same as Lemma \ref{quadsplits} except that $v=s-r+2$ must be added to case 2.
So all $m$-actions split according to the rules laid out in \cite[Lemma 5]{BC12} and 
$\tilde{m}$-actions split according to Lemma \ref{quadsplits}. Also,  these splits are independent (in the sense that the splitting of $m_i$ does not affect the splitting of $\tilde{m}_j$). By repeating the analysis above and using the splitting rules from \cite{BC12} and Lemma \ref{quadsplits}, we can derive a truth table recursion for the $g$ functions. 

Note that, each time the truth table splits and additional two $m$-actions are produced on the right half, $m _{s} $ and $\tilde{m} _{a}$ in addition to any $m$-actions that may already be present. 

We can represent a collection of $m$-actions using the {\it operation } notation introduced in \cite[Definition 6]{BC12}: For any $0\leq i \leq 2^{s-1}$, let $(i)_2$ be the string of $s$ $1$'s and $0$'s that make up the base 2 representation of $i$, prepending with 0's if necessary. Then we define $O_i$ be the collection of $m$-actions 
$$
O_i = \{ m_j | \text{ the }j \text{th digit (counting from the left) of }(i)_2 \text{ is 1} \}.  
$$   
Let $\tilde{O}_i $ be defined analogously for the $\tilde{m} $-actions. That is, 
$$
\tilde{O}_i = \{ \tilde{m}_j | \text{ the }j \text{th digit (counting from the left) of }(i)_2 \text{ is 1}. \}  
$$

We can develop a recursion for the truth table of $g$ by repeating the method outlined above (similar to the method employed in  \cite[Section 4]{BC12}), using 
\cite[Lemma 5]{BC12} to determine how the $m$ and $\tilde{m}$-actions split. We can keep track of all of the operation splits for both $O$ and $\tilde{O}$ operations using the $\Omega $ operation defined by
\begin{align*}
\Omega _i = \{ &m_j | \text{ the }j \text{th digit (counting from the left) of }(i)_2 \text{ is 1} \}\cup \\
& \{\tilde{m}_k | \text{ the }(k+s) \text{th digit (counting from the left) of }(i)_2 \text{ is 1} \}.
\end{align*}
By using  \cite[Lemmas 5 and 7]{BC12}, we can create a $2^{s+a-2}+1 \times 2^{s+a-2}+1$ matrix $A$, whose $i$th column represents the $\Omega $ operations that are produced when $\Omega _i$ splits. Now \cite[Th. 1, p. 116]{BC12} implies that the minimal polynomial for $A$ gives a recursion for the weights of $g$. 

\subsection{Recursion for the weights of $(1,r,s)_n + (1,a)_n$}
\label{subsec2.2}
In this section results from \cite{BC12} are frequently used, and the reader may consult that paper if more details are needed. Define $f^n = (1,r,s)_n + (1,a)_n,$ so  $f^n=f_1 + \tilde{f}_2$ in the notation of Section 
\ref{subsec2.1}.
Now \cite[Th. 3, p. 125]{BC12} for cubic MRS functions when applied to $f_1$ states that the weights for $f_1$ satisfy the same recursion as can be produced using the method above for $g_1$. To see that this also holds for not only quadratic functions, but for our sum functions as well, we consider the following:\\
First note that we can write $f^n$ as
\begin{align*}
f^n = f_1^n + \tilde{f}_2^n &= x_1x_{r}x_{s}+ \ldots + x_{1+n-s}x_{r+n-s}x_n \\
&+ x_1x_{2+n-s}x_{r+1+n-s}+\ldots + x_{n-r+1}x_{s-r}x_n \\
&+ x_{n-r+2}x_1x_{s-r+1}+\ldots + x_nx_{r-1}x_{s-1}\\
&+ x_1x_a+ \ldots + x_{1+n-a}x_n \\
&+ x_1x_{2+n-a}+\ldots + x_{1+a-2}x_n \\
&=g_{r,s}+g_{r',s'}+g_{r'',s''}+\tilde{g}_{a}+\tilde{g}_{2+n-a}
\end{align*}
Since $g_{r',s'}, g_{r'',s''}$, and $\tilde{g}_{2+n-a}$ depend on $n$, we cannot use the same techniques above to find the recursion. Instead, we look at the individual monomials that make up each of the above $g$ and $\tilde{g}$ functions. The work for the $g$ functions has been done in \cite[Section 5]{BC12}. The results are as follows: Let the monomials from $g_{r',s'}$ be given by $h_i=x_{i}x_{n-s+1+i}x_{n-s+r+i}$. Then \cite[Lemma 10, p. 117]{BC12} gives the truth table recursion for $h_j^n$ as
\begin{align}
T(h^n_j) &= \left( T^1_j(h^{n-1}_j) \| T^1_j(h^{n-1}_j) \| T^2_j(h^{n-1}_j) \| T^2_j(h^{n-1}_j) \right) _{2^{j-1}} \notag \\
&\text{or, equivalently} \notag \\
&= T^1_j(h^{n-1}_j) \| T^1_j(h^{n-1}_j) \| T^2_j(h^{n-1}_j) \| T^2_j(h^{n-1}_j) \| \ldots T^{2^j} _j(h^{n-1}_j) \| T^{2^j}_j(h^{n-1}_j) \notag
\end{align}
where the notation $T^a_b( u)$ represents the $a$-th portion of the truth table for ``u" after it has been divided into $2^b$ equally sized pieces. Note that the above recursion gives the truth table for $h_j^n$ as the concatenation of $2^j$ pairs, or $2^{j+1}$ pieces. Further, \cite[Remark 3, p. 118]{BC12} states that for $1\leq j \leq s-r$, the above truth table recursions can be further broken down into $2^k$ pieces for $j+1 \leq k \leq n-s+j$:
\begin{align}
T(h_j^n) &= \left( (T^1_j(h_j^{n-k-j}))_{2^{k-j}} \| (T^2_j(h_j^{n-k-j}))_{2^{k-j}} \right) _{2^{j-1}} \notag \\
& \text{or, equivalently} \notag \\
T(h_j^n) &=  \underbrace{T^1_j(h_j^{n-k-j}) \| \ldots \| T^1_j(h_j^{n-k-j})}_{2^{k-j} \text{ times}} \| \ldots \| \underbrace{T^{2^j}_j(h_j^{n-k-j}) \| \ldots \| T^{2^j}_j(h_j^{n-k-j})}_{2^{k-j} \text{ times}} \notag
\end{align}
Note that, above, by choosing large $n$, we can make $k$ as large as we please. A similar analysis can be done on the monomials, $\eta$ from $g_{r'',s''}$. For $1\leq j \leq r-1$, the recursion is 
\begin{align}
T(\eta_j^n) &= \left( \left( \left( T^1_{s-r+j}(\eta_j^{n-1}) \right) _2 \right) _{2^{s-r}} \| \left( \left( T^{2^{s-r}+1}_{s-r+j}(\eta_j^{n-1}) \right) _2 \| \left(T^{2^{s-r}+2}_{s-r+j}(\eta_j^{n-1}) \right) _2 \right) _{2^{s-r-1}} \right) _{2^{j-1}} \notag \\
& \text{or, equivalently } \notag \\
&= T^1_{s-r+j}(\eta_j^{n-1}) \| T^1_{s-r+j}(\eta_j^{n-1}) \| \ldots \| T^{2^{s-r+j}}_{s-r+j}(\eta_j^{n-1}) \| T^{2^{s-r+j}}_{s-r+j}(\eta_j^{n-1}) \notag
\end{align}
which splits $T(\eta_j^n)$ into $2^{s-r+j}$ pairs ($2^{s-r+j}$ parts). In these recursions, too, the truth tables can be broken up into arbitrarily small pieces, depending on the size on $n$. For $1\leq j \leq r-1$, we can split $T(\eta_j^n)$ into $2^k$ equal portions with $s-r+j+1 \leq k \leq n-r+j$: 
\begin{align}
T(\eta_j^n) &= (A_{2^{s-r}} \| B_{2^{s-r-1}}) _{2^{j-1}} \text{ where, } \notag \\
A &= \left( T^1_{s-r+j}(\eta_j^{n-k-r+s+1})\right) _{2^{k-s+r-j}} \text{ and} \notag \\
B &= \left( T^{2^{s-r}+1}_{s-r+j}(\eta_j^{n-k-r+s+1})\right) _{2^{k-s+r-j}} \| \left( T^{2^{s-r}+2}_{s-r+j}(\eta_j^{n-k-r+s+1})\right) _{2^{k-s+r-j}}  \notag \\
& \text{or, equivalently} \notag \\
T(\eta_j^n) &= \underbrace{T^1_{s-r+j}(\eta_j^{n-k-r+s+1}) \| \ldots  T^1_{s-r+j}(\eta_j^{n-k-r+s+1})}_{2^{k-s+r-j} \text{ times}} \ldots \notag \\
& \hspace{3pc} \ldots \| \underbrace{T^{2^{s-r+j}}_{s-r+j}(\eta_j^{n-k-r+s+1}) \| \ldots \| T^{2^{s-r+j}}_{s-r+j}(\eta_j^{n-k-r+s+1})}_{2^{k-s+r-j} \text{ times}} \notag
\end{align}
On the other hand, the $a-1$ quadratic monomials from $\tilde{g}_{2+n-a}$ all take the form $\tilde{h}_i=x_ix_{n-a+1+i}$. \\
\begin{lemma}
\label{quadTTgeneral}
The truth table of $x_ix_j$ in $n$ variables where $1\leq i \leq j \leq n$ is given by
$$
(0_{2^{n-i}}(0_{2^{n-j}}1_{2^{n-j}})_{2^{j-i-1}})_{2^{i-1}}
$$
\end{lemma} 
\begin{proof}
This result (in a different notation) was given in \cite[Lemma 8, p. 293]{CS02}.  It is a very special case of Lemma \ref{generalTT}.
\end{proof}
\begin{lemma}
\label{htilde}
The truth table of $\tilde{h}_j=x_jx_{n-a+1+j}$ for $1\leq j \leq a-1$ is given by
\begin{align}
T(\tilde{h}^n_j) &= \left( T^1_j(\tilde{h}^{n-1}_j) \| T^1_j(\tilde{h}^{n-1}_j) \| T^2_j(\tilde{h}^{n-1}_j) \| T^2_j(\tilde{h}^{n-1}_j) \right) _{2^{j-1}} \label{eq1} \\
&\text{or, equivalently} \notag \\
&= T^1_j(\tilde{h}^{n-1}_j) \| T^1_j(\tilde{h}^{n-1}_j) \| T^2_j(\tilde{h}^{n-1}_j) \| T^2_j(\tilde{h}^{n-1}_j) \| \ldots T^{2^j} _j(\tilde{h}^{n-1}_j) \| T^{2^j}_j(\tilde{h}^{n-1}_j) \notag
\end{align}
where the notation $T^a_b( - )$ is as defined above.
\end{lemma}
\begin{proof}
This follows easily using the same method used to prove \cite[Lemma 10, p. 117]{BC12}.
\end{proof}
As was the case with the $h_j$ monomials, we can further split the truth tables for the $\tilde{h}_j$ monomials above into $2^k$ pieces for $j+1 \leq k \leq n-a+j$:
\begin{align}
T(\tilde{h}_j^n) &= \left( (T^1_j(\tilde{h}_j^{n-k-j}))_{2^{k-j}} \| (T^2_j(\tilde{h}_j^{n-k-j}))_{2^{k-j}} \right) _{2^{j-1}} \label{eq2} \\
& \text{or, equivalently} \notag \\
T(\tilde{h}_j^n) &=  \underbrace{T^1_j(\tilde{h}_j^{n-k-j}) \| \ldots \| T^1_j(\tilde{h}_j^{n-k-j})}_{2^{k-j} \text{ times}} \| \ldots \| \underbrace{T^{2^j}_j(\tilde{h}_j^{n-k-j}) \| \ldots \| T^{2^j}_j(\tilde{h}_j^{n-k-j})}_{2^{k-j} \text{ times}} \notag
\end{align}
Again, by choosing large $n$, we can make $k$ as large as we please.

Considering them individually, we see that the recursions for $g_{r',s'}$ and $\tilde{g}_{2+n-a}$ each split the truth tables into a maximum of $2^{s-r+1}$ or $2^{a}$ pieces, respectively. On the other hand, the recursion for $g_{r'',s''}$ splits the truth table into $2^{s}$ pieces. Thus, to make the recursions for $f$ uniform, we write split the recursions for $g$ and the remaining monomials of $g_{r',s'}$,  $g_{r'',s''}$, $g_{2+n-a}$, into $2^{\text{max}\{s,a\}}$ pieces using the equations above. Thus, letting $s = \text{max}\{s1,s2\}$, we see that the truth table recursion for $f$ is given by
\begin{align}
T(f^n) &= \overset{\Omega _{w_1}}{T(g^{n-s}) \bigoplus _{i=1}^{s-r} T^1_s(h_i^n) \bigoplus _{i=1}^{r-1} T^1_s(\eta _i^n) \bigoplus _{i=1}^{a} T^1_s(\tilde{h}_i^n)}\| \ldots \notag \\ 
& \hspace{3pc} \ldots \| \overset{\Omega _{w_{2^{s}}}}{T(g^{n-s}) \bigoplus _{i=1}^{s-r} T^{2^s}_s(h_i^n) \bigoplus _{i=1}^{r-1} T^{2^s}_s(\eta_i^n) \bigoplus _{i=1}^{a} T^{2^s}_s(\tilde{h}_i^n) }\label{TT}
\end{align}
where the $\Omega $ operations above are those produced from the splitting of $T(g)$ (since no $m$-actions are produced when any of the monomials $h$ or $\tilde{h}$ split). 
\begin{remark}
In order for the above recursion to work, we assume that, for $1\leq j \leq a-1$, 
$\tilde{h}_j=x_jx_{n-a+1+j}$ exists for all $n$. However, in the case where $n=2(a-1)$ (here we would have a "short" quadratic function; see \cite[p. 430]{Kim09}), the monomials of $\tilde{h}$ are of the form
$$
x_jx_{2(a-1)-a+1+j}=x_jx_{2a-2-a+1+j}=x_jx_{a-1+j}
$$
which are exactly the same monomials as are in $\tilde{g}$. Thus the monomials $\tilde{h}_j$ cancel out $\tilde{g}$ and all we are left with, for these values of $n$ is $f^n=f_1$. Therefore, we redefine our function $f$ to be
$$
\bar{f} = \left\{ \begin{array}{rl}
	f_1+\tilde{f}_2 &\mbox{ if $n\neq 2(a-1)$} \\
	f_1 &\mbox{ otherwise}
	\end{array} \right.
$$
where $f_1$ is the cubic MRS function generated by $x_1x_rx_s$ and $\tilde{f}_2$ is the MRS quadratic MRS function generated by $x_1x_a$.
\end{remark}
We note that we can pair each of the $2^s$ portions of the truth table in \eqref{TT} so that for $j$ an odd integer, the portions $j$ and $j+1$ have the same structure (see 
\cite[Th. 2, p. 121]{BC12}). Thus, the arguments given in  \cite[Lemma 13 and Th. 3]{BC12} also apply to the recursions for  $T(g^n)$ and $T(\bar{f}^n)$, so the weights of $\bar{f}^n$ satisfy the recursion for the weights of $g^n$. 

\section{Weight recursions for any rotation symmetric function}
\label{wtrecur}
In this section we generalize the work of Section \ref{deg3+deg2} to any rotation symmetric function, that is, to any sum of MRS functions.  We give the details of extending the definition of the actions in Section \ref{subsec2.1} to functions of degree greater than $3,$ but due to the elaborate notation that would be required, we do not give the details for the application of these actions to the computation of the weight recursions for rotation symmetric functions in general. It suffices to follow the method in Section \ref{subsec2.2}, as is done in the Mathematica program in Section \ref{prog} at the end of this paper. 

In the following discussion, unless otherwise noted, we will assume that the function in $n$ variables given by the monomial $x_{k_1}x_{k_2}\ldots x_{k_m}$ satisfies  $k_1 < k_2 < \ldots < k_m$ and  $k_i \leq n$ for all $i$.
\begin{lemma}
\label{generalTT}
Let the truth table for the monomial $x_{k_1}x_{k_2}\ldots x_{k_m}$ in $n$ variables be denoted $T^n([k_1, \ldots, k_m])$. Then $T^n([k_1, \ldots, k_m])$ is given by
$$0_{2^{n-k_1}}(0_{2^{n-k_2}}(\ldots (0_{2^{n-k_{m-1}}}(0_{2^{n-k_m}}1_{2^{n-k_m}})_{2^{k_m-k_{m-1}-1}})_{2^{k_{m-1}-k_{m-2}-1}})\ldots )_{2^{k_2-k_1-1}})_{2^{k_1-1}} $$
\end{lemma}
\begin{proof}
A similar result is given in \cite[Lemma 11, p. 298]{CS02} in a different notation, but the proof is omitted. We provide a proof here, using more convenient notation, as in Lemma \ref{quadTTgeneral}. We proceed by induction. From \cite[Lemma 8, p. 293]{CS02} and \cite[Lemma 9, p. 117]{BC12}, we have
\begin{align*}
T^n(x_{k_1},x_{k_2}) &= (0_{2^{n-k_1}}(0_{2^{n-k_2}}1_{2^{n-k_2}})_{2^{k_2-k_1-1}})_{2^{k_1-1}}\\
T^n(x_{k_1},x_{k_2},x_{k_3}) &= (0_{2^{n-k_1}}(0_{2^{n-k_2}}(0_{2^{n-k_3}}1_{2^{n-k_3}})_{2^{k_3-k_2-1}})_{2^{k_2-k_1-1}})_{2^{k_1-1}}
\end{align*}
Assume 
\begin{align*}
T^n(&[k_1, \ldots, k_r]) = \\
&0_{2^{n-k_1}}(0_{2^{n-k_2}}(\ldots (0_{2^{n-k_{r-1}}}(0_{2^{n-k_r}}1_{2^{n-k_r}})_{2^{k_r-k_{r-1}-1}})_{2^{k_{r-1}-k_{k-2}-1}})\ldots )_{2^{k_2-k_1-1}})_{2^{k_1-1}}
\end{align*}
for all $k_1, \ldots , k_r$. To find the truth table for $x_{k_1}x_{k_2}\ldots x_{k_{r+1}}$ in $n$ variables, we begin by noticing that the truth table for the product of the last $k$ variables (i.e. $x_{k_2}x_{k_3}\ldots x_{k_{r+1}}$) is the same as the truth table for $x_{k_2-k_1}\ldots x_{k_{r+1}-k_1}$ in $n-k_1$ variables repeated $2^{k_1}$ times. From our induction hypothesis, we have
\begin{align*}
T^{n-k_1}(&[k_2-k_1, \ldots, k_r-k_1]) = \\
&0_{2^{n-k_1-(k_2-k_1)}}(\ldots (0_{2^{n-k_1-(k_{r-1}-k_1)}}(0_{2^{n-k_1-(k_r-k_1)}}1_{2^{n-k_1-(k_r-k_1)}})_{2^{k_r-k_1-(k_{r-1}-k_1)-1}})\ldots \\
&\hspace{2cm} \ldots )_{2^{k_3-k_1-(k_2-k_1)-1}})_{2^{k_2-k_1-1}}\\
&=0_{2^{n-k_2}}(0_{2^{n-k_3}}(\ldots (0_{2^{n-k_{r-1}}}(0_{2^{n-k_r}}1_{2^{n-k_r}})_{2^{k_r-k_{r-1}-1}})_{2^{k_{r-1}-k_{k-2}-1}})\ldots )_{2^{k_2-k_1-1}}
\end{align*}
This implies
\begin{align*}
T^n(&[k_2, \ldots ,k_{r+1}]) = \\
&0_{2^{n-k_2}}(0_{2^{n-k_2}}(\ldots (0_{2^{n-k_{r-1}}}(0_{2^{n-k_{r+1}}}1_{2^{n-k_{r+1}}})_{2^{k_{r+1}-k_r-1}})_{2^{k_r-k_{r-1}-1}})\ldots )_{2^{k_2-k_1-1}})_{2^{k_1}}
\end{align*}
When we multiply the above function by $x_{k_1}$ (to get our original monomial $x_{k_1}x_{k_2}\ldots x_{k_{r+1}}$), we see that the monomial can only be nonzero when $x_{k_1}$ is nonzero. In particular, this means that the truth table of $x_{k_1}x_{k_2}\ldots x_{k_{r+1}}$ is 0 for the first $2^{n-k_1}$ entries. So we have
\begin{align*}
T^n(&[k_1, \ldots, k_r,k_{r+1}]) = \\
&0_{2^{n-k_1}}(0_{2^{n-k_2}}(\ldots (0_{2^{n-k_{r-1}}}(0_{2^{n-k_{r+1}}}1_{2^{n-k_{r+1}}})_{2^{k_{r+1}-k_r-1}})_{2^{k_r-k_{r-1}-1}})\ldots )_{2^{k_2-k_1-1}})_{2^{k_1-1}}
\end{align*}
as desired.

\end{proof}
\begin{corollary}\label{quartTT1}
The truth table for $x_1x_ix_jx_k$ in $n$ variables is given by
$$
T^n([1,i,j,k]) = 0_{2^{n-1}}(0_{2^{n-i}}(0_{2^{n-j}}(0_{2^{n-k}}1_{2^{n-k}})_{2^{k-j-1}})_{2^{j-i-1}})_{2^{i-2}}
$$
\end{corollary}
To find the recursions for weights of quartic MRS functions, we begin by defining the quartic analog of $m$-actions:
\begin{definition}
For $v\in \mathbb{Z}$, $1\leq v \leq k$, we define an \emph{$\hat{m} _v $-action} on a truth table of a quartic Boolean function $f^{n-1}$ in $n-1$ variables as follows: Define $\hat{m} _v$ to be the sequence obtained by splitting the truth table of $x_1x_ix_jx_k$ into $2^{k-v+1}$ equally-sized portions, isolating the final portion, and then stretching it (by repeating each entry $2^{k-v}$ times) to a length of $2^{n-1}$. Then the $\hat{m}_v$-action applied to $f^{n-1}$ is
$$
T(\overset{\bar{m}_v}{f^{n-1}}) = T(f^{n-1}) \oplus \hat{m}_v
$$ 
where $T(f^{n-1})$ is the truth table of $f^{n-1}$. 
\end{definition}
\begin{lemma}\label{quadactions}
For $1\leq v\leq k$, the $\hat{m}$-actions are given by 

\begin{align*}
1.\hspace{.5pc} & \hat{m} = 1_{2^{n-1}} &v=1\\
2.\hspace{.5pc} & \hat{m} = (0_{2^{n-v}}1_{2^{n-v}})_{2^{v-2}}  & 1< v \leq k-j+1\\
3.\hspace{.5pc} & \hat{m} = (0_{2^{n+k-j-v}}(0_{2^{n-v}}1_{2^{n-v}})_{2^{k-j-1}})_{2^{v+j-k-2}}  & k-j+1< v \leq k-i+1\\
4.\hspace{.5pc} & \hat{m} = (0_{2^{n-i+k-v}}(0_{2^{n+k-j-v}}(0_{2^{n-v}}1_{2^{n-v}})_{2^{k-j-1}})_{2^{j-i-1}})_{2^{v+i-k-2}}  & k-i+1< v \leq k
\end{align*}
\end{lemma}
We can extend this definition to the general case, that is, $m$-actions on a function of degree $k$ (called $\mu ^k$ actions below).
\begin{definition}[General $m$-actions]\label{genm}
For $v\in \mathbb{Z}$, $1\leq v \leq k$, we define an \emph{$\mu ^k_v $-action} on a truth table of a degree $k$ Boolean function $f^{n-1}$ in $n-1$ variables as follows: Define $\mu ^k _v$ to be the sequence obtained by splitting the truth table of $x_1x_{k_1}\ldots x_{k_r}$ into $2^{k_r-v+1}$ equally-sized portions, isolating the final portion, and then stretching it (by repeating each entry $2^{k_r-v}$ times) to a length of $2^{n-1}$. Then the $\mu ^k_v$-action applied to $f^{n-1}$ is
$$
T(\overset{\mu^k_v}{f^{n-1}}) = T(f^{n-1}) \oplus \mu^k_v
$$ 
where $T(f^{n-1})$ is the truth table of $f^{n-1}$. 
\end{definition}
\begin{align*}
T^n(&[k_1, \ldots, k_m]) = \\
&0_{2^{n-k}}(0_{2^{n-k_2}}(\ldots (0_{2^{n-k_{r-1}}}(0_{2^{n-k_r}}1_{2^{n-k_r}})_{2^{k_r-k_{r-1}-1}})_{2^{k_{r-1}-k_{r-2}-1}})\ldots )_{2^{k_2-k_1-1}})_{2^{k_1-1}}
\end{align*}
\begin{lemma}\label{muactions}
The for $1\leq v\leq k_r$, $\mu ^k$-actions are given by 

\begin{align*}
1.\hspace{.5pc} & \mu ^k = 1_{2^{n-1}} & v=1\\
2.\hspace{.5pc} & \mu ^k= (0_{2^{n-v}}1_{2^{n-v}})_{2^{v-2}} & 1< v \leq k_r-k_{r-1}+1 \\
3.\hspace{.5pc} & \mu ^k= (0_{2^{n+k_r-k_{r-1}-v}}(0_{2^{n-v}}1_{2^{n-v}})_{2^{k_r-k_{r-1}-1}})_{2^{v-k_r+k_{r-1}-2}} & k_r-k_{r-1}+1< v \leq k_r-k_{r-2}+1 \\
& \vdots &\vdots \qquad \qquad
\end{align*}
In general, for $0\leq j<k$ and $k_r - k_{k-j} +1 < v \leq k_r -k_{k-j-1} +1$ (recall $k_0 =1$), we have
$$
\mu ^k = (0_{2^{n+k_r-k_{k-j}-v}}(0_{2^{n+k_r-k_{k-j+1}-v}}(\ldots (0_{2^{n-v}}1_{2^{n-v}})_{2^{k_r-k_{r-1}-1}})_{2^{k_{r-1}-k_{r-2}-1}})\ldots)_{2^{v-k_r+k_{k-j}-2}}
$$
\end{lemma}

\begin{proof}
This follows from Lemma \ref{generalTT}.
\end{proof}
The next step in deriving the recursions is determining how the $m$-actions split. We will start with the quartic case and then generalize. 
\begin{lemma}\label{4splits}
For any quartic $\hat{m}$-action, $\hat{m}_v$ where $k<n$ and $1\leq v \leq k$, we can split $\hat{m}_v$ into two $\hat{m}$-actions, $\hat{m}_{v1}$ and $\hat{m}_{v2}$ such that $\hat{m}_{v1}$ acts on the first half of a truth table
$T(b^{n-1})$ of a Boolean function in $n-1$ variables and $\hat{m}_{v2}$ acts on the second half as follows: \begin{align*}
1.\hspace{.5pc} & \overset{\hat{m}_v^n}{T(b^{n-1})} = \overset{\hat{m}_v^{n-1}}{T^1(b^{n-1})}\| \overset{\hat{m}_v^{n-1}}{T^2(b^{n-1})} &v=1\\
2.\hspace{.5pc} & \overset{\hat{m}_v^n}{T(b^{n-1})} = T^1(b^{n-1})\| \overset{\hat{m}_{v-1}^{n-1}}{T^2(b^{n-1})} &v=2, \\
&&v=k-j+2, \text{ or }\\
&& v=k-i+2\\
3.\hspace{.5pc} & \overset{\tilde{m}_v^n}{T(b^{n-1})} = \overset{\tilde{m}_{v-1}^{n-1}}{T^1(b^{n-1})}\| \overset{\tilde{m}_{v-1}^{n-1}}{T^2(b^{n-1})} &otherwise\\
\end{align*}
where $T^1(b^{n-1})$ is the first half of $T(b^{n-1})$ and $T^2(b^{n-1})$ is the second half.
\end{lemma}
The general splits are along the same lines.
\begin{lemma}\label{gensplits}
For any $\mu_{k,v}$-action, $\mu_{k,v}$ where $k_r<n$ and $1\leq v \leq k_r$, we can split $\mu_{k,v}$ into two $\mu_{k,v}$-actions, $\mu_{k,v1}$ and $\mu_{k,v2}$ such that $\mu_{k,v1}$ acts on the first half of a truth table 
$T(b^{n-1})$ of a Boolean function in $n-1$ variables and $\mu_{k,v2}$ acts on the second half as follows: 
\begin{align*}
1.\hspace{.5pc} & \overset{\mu_{k,v}^n}{T(b^{n-1})} = \overset{\mu_{k,v}^{n-1}}{T^1(b^{n-1})}\| \overset{\mu_{k,v}^{n-1}}{T^2(b^{n-1})} &v=1\\
2.\hspace{.5pc} & \overset{\mu_{k,v}^n}{T(b^{n-1})} = T^1(b^{n-1})\| \overset{\mu_{k,v-1}^{n-1}}{T^2(b^{n-1})} &v=2 \text{ or} \\
&&v=k_r-k_{k-j}+2\\
&&\text{for some }0<j<k\\
3.\hspace{.5pc} & \overset{\tilde{m}_v^n}{T(b^{n-1})} = \overset{\tilde{m}_{v-1}^{n-1}}{T^1(b^{n-1})}\| \overset{\tilde{m}_{v-1}^{n-1}}{T^2(b^{n-1})} &otherwise\\
\end{align*}
where $T^1(b^{n-1})$ is the first half of $T(b^{n-1})$ and $T^2(b^{n-1})$ is the second half.
\end{lemma}

\section{Mathematica program for weight recursions}
\label{prog}

Below are the lines of Mathematica code in a program which computes the recursion polynomial (as in Example \ref{ex1}) for the sequence of weights $w_k = wt(f_k)$ for any rotation symmetric functions  $f_k$ in $k$ variables defined by adding $n$ fixed MRS functions of varying degrees.   Of course $w_k$ is only defined if $k$ is at least as large as the minimal degree of the MRS functions, but $w_k$ could be defined by working the recursion backwards even for smaller values of $k \geq 0.$  It turns out to be useful to use this idea in \cite{C13}.    

This program [January 2017 version] computes the recursion polynomial for the sequence of Hamming weights $wt(f_m)$ where $f_m$ is the sequence, as $m$ increases, of rotation symmetric (RS) functions in $m$ variables generated by any given sum of $n$ monomial RS functions with arbitrary degrees. The program uses an extension of the theory (see \cite{BC12}) for the case of a single RS function. The Òrules matrixÓ constructed in the program is explained in 
\cite[pp. 112-115]{BC12}, where the matrix is labelled $A.$ The program below has many comments to explain what is being done at various points. After the recursion polynomial is found, there is another program which can be used to find $wt(f_m)$ for small $m$ in order to obtain initial conditions for the recursion.  The program was first written by Cusick's Ph. D. student Bryan Johns in 2013. Note that the program below must be entered into Mathematica software to produce a .nb file which can then be run. \\
\subsection{Mathematica code} \label{4.1}
Clear[rules];\\

(* ENTER "n = number of functions that occur in f;" *)\\
n = 3;\\

(* ENTER the list of generating monomials for each of the monomial functions that occur in f using the form \{1,x2,x3,...,xn\} if the generating monomial is $[1,x2,x3,...,xn].$ Note all monomials MUST begin with 1. The list is named fcns so a sample list is "fcns = \{ \{1,2,3,4\},\{1,2,5\},\{1,4\} \};Ó *)
\\  
fcns = \{ \{1, 2, 6\}, \{1, 2\}, \{1, 6\} \};\\

(* The main program begins here. *)\\
ctr = 0;\\
ctr2 = 0;\\
error = 0;\\

(* checking the form of the generating monomials *)\\
For[i = 1, i $\leq$ n, i++, 
  If[fcns[ [i] ][ [1] ] == 1 \&\& VectorQ[fcns [ [i] ] ], , \\
   MessageDialog[
    "All generating monomials must be entered in the correct form, beginning with 1"]; error += 1] ];\\

(* checking for the case of linear function only *)\\
If[n == 1 \&\& fcns == \{ \{1\} \}, \\
  Print["The recursion for the linear RS function x1+x2+...+xn is 2"];\\
   error += 1];\\
fcns2 = fcns;\\

(* checking for the presence of linear and quadratic functions *)\\
For[i = 1, i $\leq$ n, i++, If[Length[fcns[ [i] ] ] == 2, ctr2 += 1] ];\\
For[i = 1, i $\leq$ n, i++, If[Length[fcns[ [i] ] ] == 1, ctr += 1; fcns = Drop[fcns, \{i\}]; 
   n = n - 1] ];\\

(* this begins the construction of the rules matrix *)\\
If[error == 0,\\
 Rs = 0;\\
 For[i = 1, i $\leq$ n, i++, Rs += fcns[ [i] ][ [-1] ] ];\\
 rules = ConstantArray[0, \{2$\symbol{94}$(Rs - n) + 1, 2$\symbol{94}$(Rs - n) + 1\}];\\
 For[j = 1, j $\leq$ 2$\symbol{94}$(Rs - n), j++,\\
  lft = \{ \};\\
  rt =\{ \};\\
  testcase = 1;\\
  bits = IntegerDigits[j - 1, 2, Rs - n];\\
  start = 1;\\
  end = fcns[ [1] ][ [-1] ];\\
  For[k = 1, k $\leq$ n, k++,\\
   ik = fcns[ [k] ][ [-1] ];\\
   fkb = FromDigits[bits[ [start ;; end - 1] ], 2];\\
   If[k == n, ,\\
    start = end;\\
    end = end + fcns[ [k + 1] ][ [-1] ] - 1];\\
   
(* the left side of the rules matrix is computed *)\\
   fkbits = IntegerDigits[fkb, 2, ik];\\
   If[fkb == 2$\symbol{94}$(ik - 1), fkbits = 2*fkb];\\
   If[fkbits[ [2] ] == 1, lftfk = 2*fkb - 2$\symbol{94}$(ik - 1), lftfk = 2*fkb];\\
   If[Length[fcns[ [k] ] ] $>$ 2, For[i = 2, i $<$ Length[fcns[ [k] ] ], i++,\\
     If[fkbits[ [ik - fcns[ [k] ][ [i] ] + 2] ] == 1, 
      lftfk = lftfk - 2$\symbol{94}$(fcns[ [k] ][ [i] ] - 1)] ] ];\\
   
(* the right side of the rules matrix is computed *)\\
   Which[fkb $<$ 2$\symbol{94}$(ik - 2), rtfk = 2*fkb + 1, 
    2$\symbol{94}$(ik - 2) $\leq$ fkb $<$ 2$\symbol{94}$(ik - 1), rtfk = -(2*fkb + 1 - 2$\symbol{94}$(ik - 1)), 
    fkb = 2$\symbol{94}$(ik - 1), rtfk = 2*fkb];\\
   fklft = IntegerDigits[lftfk, 2, ik - 1];\\
   lft = Join[lft, fklft];\\
   leftop = FromDigits[lft, 2];\\
   fkrt = IntegerDigits[Abs[rtfk], 2, ik - 1];\\
   rt = Join[rt, fkrt];\\
   rtop = FromDigits[rt, 2];\\
   testcase = testcase*(rtfk + .25)];\\
  
(* compiling the rules matrix, beginning with case where there are no linear functions present *)\\
   If[ctr == 0,\\
   If[testcase $>$= 0, rules[ [rtop + 1] ][ [j] ] = 1, 
    rules[ [rtop + 1, j] ] = -1;\\
    rules[ [-1] ][ [j] ] = 1];\\
   rules[ [leftop + 1] ][ [j] ] = 1;\\
   Clear[lftfk],\\
   If[testcase $<$ 0, rules[ [rtop + 1] ][ [j] ] = 1, 
    rules[ [rtop + 1, j] ] = -1;\\
    rules[ [-1] ][ [j] ] = 1];\\
   rules[ [leftop + 1] ][ [j] ] = 1;\\
   Clear[lftfk]
   ] ];\\

 (* This MatrixMinimalPolynomial code was taken from the Wolfram website at 
    http://mathworld.wolfram.com/MatrixMinimalPolynomial.html *)\\
  MatrixMinimalPolynomial[a\_List?MatrixQ, x\_] := 
  Module[\{i, n = 1, qu =\{ \}, 
    mnm = \{Flatten[IdentityMatrix[ Length[a] ] ]\} \},
   While[Length[qu] == 0,
    AppendTo[mnm, Flatten[ MatrixPower[a, n] ]
     ];\\
    qu = NullSpace[ Transpose[mnm] ];\\
    n++;\\
    ];\\
   First[qu].Table[x$\symbol{94}$i, \{i, 0, n - 1\}]\\
   ];\\
 rules[ [-1] ][ [-1] ] = 2;\\
 j = Length[rules];\\
 For[i = 1, i $\leq$ j, i++, 
  If[ rules[ [i] ] == 
    ConstantArray[0, Length[rules[ [i] ] ] ], \{rules = 
     Drop[rules, \{i, i\}, \{i, i\}], j = Length[rules], i = 1\}] ];\\
 If[ctr == 0,\\
  Print["Minimal Polynomial of the rules matrix for the sum of ", fcns, ":"],\\
  Print["Minimal Polynomial of the rules matrix for the sum of ",fcns, "+ $\Sigma x_i$:"]\\
   ];\\
 minpoly = MatrixMinimalPolynomial[rules, x];\\
 Print[minpoly];\\

 (*  collecting information for the computation of weights using the recursion computed above *)\\
 minpoly2 = minpoly;\\
 While[ CoefficientList[minpoly2, x][ [1] ] == 0,  minpoly2 = 1/x*minpoly2];\\
 minpoly2 = Simplify[minpoly2];\\
 degminpoly = Exponent[minpoly2, x];\\
 kerback = CoefficientList[minpoly2, x];\\ 
 kerback = Drop[kerback, -1];\\
 ker = Reverse[kerback];\\
 ker = -1*ker;\\
 ivals = \{ \};\\
 large = 1;\\
 
(* Next we remove extra powers of $x$ in the minimal polynomial *)\\ 
 Print["Which reduces to ", minpoly2];\\  
 ]\\
(* The main program ends here and an output is printed. *)\\

(* The next three lines are output from the main program. *)\\
Minimal Polynomial of the rules matrix for the sum of \{ \{1,2,6\},\{1,2\},\{1,6\} \}:

$$-8 x^9+4 x^{10}+4 x^{11}+2 x^{12}-2 x^{13}-2 x^{14}+x^{15}$$

Which reduces to $-8+4 x+4 x^2+2 x^3-2 x^4-2 x^5+x^6$

(* Below is a program which uses results from the main program to compute weights for the function. *)\\

(* If you want to compute some weights of this function using the recursion, ENTER "yorn = 1;" and ENTER "m=number of weights to be computed;" If not ENTER "yorn=0;" or simply skip the rest of this program. *)\\
yorn = 1;\\
m = 12;\\

(* The weights program begins here. It can take a lot of time if many weights are computed. *)\\
If[yorn == 1, \\ 
 For[j = 1, j $\leq$ Length[fcns2], j++,
  If[fcns2[ [j] ][ [-1] ] $>$= large, large = fcns2[ [j] ][ [-1] ] ] ];\\
 For[i = 1, i $\leq$ Min[m, degminpoly], i++,\\
  n = i + large;\\
  fn[z\_\_] := Mod[Sum[Mod[Sum[\\
       Product[z[ [Mod[(j - 1) + k, n, 1] ]], \{j, fcns2[ [i] ]\} ], \{k, n\}], 2], \{i, Length[fcns2]\}], 2];\\
  wt = Sum[ fn[IntegerDigits[j, 2, n] ], \{j, 0, 2$\symbol{94}$n - 1\} ];\\
  ivals = Insert[ivals, wt, -1];\\
  ];\\
 If[Length[ivals] $>$= degminpoly, \\
  Tvals = LinearRecurrence[ker, ivals, m], Tvals = ivals];\\
 
(* If there are quadratic functions x1xa some may be "short;" this happens when n=2a-2. We     
    recompute these weights below. *)\\
 If[ctr2 $>$ 0,
  fcns3 = fcns2;\\
  For[i = 1, i $\leq$ Length[fcns2], i++,\\
   If[Length[fcns2[ [i] ] ] == 2 \&\& 
     1 + large $\leq$ 2*(fcns2[ [i] ][ [-1] ] - 1) $\leq$ m + large,\\
    r = fcns2[ [i] ][ [2] ];\\
    fcns3 = Drop[fcns3, \{i\}];\\
    n = 2*(r - 1);\\
    fn[z\_\_] := Mod[Sum[Mod[Sum[\\
         Product[z[ [Mod[(j - 1) + k, n, 1] ]], \{j, fcns3[ [a] ]\} ],\{k, n\}],2], \{a, Length[fcns3]\}], 2];\\
    h2[z\_\_] := 
     Mod[Sum[z[ [b] ]*z[ [Mod[r - 1 + b, n/2, 1] ] ], \{b, Length[z]\}], 2];\\
    hs[z\_\_] := Mod[fn[z] + h2[z], 2];\\
    wt2 = Sum[hs[IntegerDigits[j, 2, n] ], \{j, 0, 2$\symbol{94}$n - 1\}];\\
    Tvals = ReplacePart[Tvals, (n - large) -$>$ wt2];\\
    fcns3 = fcns2;\\
    	] ] ];\\
 
\noindent Print["The first ", m, " weights (starting with n=", large + 1, ") are ", Tvals];\\
 ] \\
(* The weights program ends here and an output is printed *)\\

(* Below is output from the weights program *)\\
The first 12 weights (starting with n=7) are\\ \{64,112,244,496,1024,1960,4096,8064,16336,32512,65536,130464\}

\subsection{Comments on the code}
The code in Section \ref{4.1} contains a special part which deals with the so-called {\em short} quadratic MRS functions in $n$ variables generated by the monomial $x_1x_a$ when $n=2a-2.$
The name comes from the fact that these MRS functions contain only $n/2$ monomials, instead of the usual $n.$ See \cite[Remark 10, p. 431]{Kim09} for the special properties of these short functions.  Without the special portion of code, the part of the program in Section \ref{4.1} which computes weights would give a wrong answer for $n = 2a-2$ whenever a short quadratic function was present in the list $fcns.$ Note that the code in Section \ref{4.1} has the short function generated by $x_1x_6$ in $fcns,$ so the weight computation for $n = 10$ uses the special part of the code.

Short functions also occur for certain values of $n$ in MRS functions of higher degree, and the number of kinds of short function increases as the degree increases.  A discussion of the short MRS functions of degrees $3$ and $4$ is given in \cite[pp. 5070-5071]{CuIS11} and 
\cite[Lemma 1.2, p. 193]{CuCh}, respectively.  The presence of higher degree short functions in the list $fcns$ will generate some incorrect weights as in the quadratic case, unless some special code is inserted to adjust for this. The sedulous reader can easily compose this code, or simply compute the correct weights for the problematic values of $n$ separately.  

The recursions computed by the Mathematica program may have large order even for relatively small values of the input parameters. For example, the order of the recursion for the weights of the cubic MRS functions $(1,3,11)_n$ is $145.$ For the simplest case of weights of cubic MRS functions $(1,r,s)_n,$ it is posible to get some general results on the values for the orders of the recursions for the weights \cite{CJ15}.  If one wanted to actually use a recursion of large order for the weights, it might be infeasible to compute the initial conditions for the recursion by examining the truth tables. For example, with the computers available now we could not count the $1's$ in a truth table of size $2^{145}.$  Perhaps surprisingly, for the case of cubic MRS functions it is possible to find the initial conditions for any recursion of weights, no matter how large the order of the recursion, if we can compute the roots of the recursion polynomial to a sufficient accuracy \cite{C13}.

\end{document}